\documentclass[11pt]{amsart}
\usepackage{amscd}
\usepackage{verbatim}
\usepackage{amssymb}
\usepackage{amsmath}
\usepackage{amsthm}
\usepackage{amsfonts}
\usepackage{mathrsfs}
\usepackage{amsthm}
\usepackage{euscript}

\usepackage[margin=30 mm,heightrounded=true,centering]{geometry}
\newtheorem{theorem}{Theorem}[section]

\newtheorem{lemma}[theorem]{Lemma}
\newtheorem{remark}[theorem]{Remark}
\newtheorem{corollary}[theorem]{Corollary}

\newcounter{mnotecount}[section]

\begin{document}

\title[Bakry-\'Emery]{Cosmological singularities in Bakry-\'Emery spacetimes}

\author{Gregory J Galloway}
\address{Dept of Mathematics, University of Miami, Coral Gables, Coral Gables, FL 33146, USA}
\email{galloway@math.miami.edu}

\author{Eric Woolgar}
\address{Department of Mathematical and Statistical Sciences,
University of Alberta, Edmonton, Alberta, Canada T6G 2G1}
\email{ewoolgar(at)ualberta.ca}

\date{\today}

\begin{abstract}
\noindent We consider spacetimes consisting of a manifold with Lorentzian metric and a weight function or scalar field. These spacetimes admit a Bakry-\'Emery-Ricci tensor which is a natural generalization of the Ricci tensor. We impose an energy condition on the Bakry-\'Emery-Ricci tensor and obtain singularity theorems of a cosmological type, both for zero and for positive cosmological constant. That is, we find conditions under which every timelike geodesic is incomplete. These conditions are given by ``open'' inequalities, so we examine the borderline (equality) cases and show that certain singularities are avoided in these cases only if the geometry is rigid; i.e., if it splits as a Lorentzian product or, for a positive cosmological constant, a warped product, and the weight function is constant along the time direction. Then the product case is future timelike geodesically complete while, in the warped product case, worldlines of certain conformally static observers are complete. Our results answer a question posed by J Case. We then apply our results to the cosmology of scalar-tensor gravitation theories. We focus on the Brans-Dicke family of theories in 4 spacetime dimensions, where we obtain ``Jordan frame'' singularity theorems for big bang singularities.
\end{abstract}

\maketitle

\section{Introduction}
\setcounter{equation}{0}

\noindent The singularity theorems of general relativity (see, e.g., \cite{HE}) are arguably some of the deepest statements in modern science. They imply that the universe has a finite history, beginning in what has come to be called a big bang singularity, provided that we assume that we can reliably extrapolate certain features of the known laws of physics back to early times and high energy scales.

Current theories of high energy physics, such as string theoretic models and theories with Kaluza-Klein dimensional reduction, postulate the presence of fundamental scalar fields such as the dilaton field, in addition to the spacetime metric. Modern cosmological models also sometimes employ scalar fields for a number of reasons. There are a variety of ways to couple scalar fields to general relativity. Scalar fields can be incorporated into the matter stress-energy tensor, as is commonly done to produce models of inflationary cosmology, or they can couple in more intricate ways, as happens with dilaton scalar fields and, more generally, in scalar-tensor gravitation theories. The prototypical examples of scalar-tensor gravitation theories are the members of the Brans-Dicke family of theories \cite{BD, Faraoni}.

There is a straightforward way to obtain singularity theorems in scalar-tensor gravitation theories. The metric of such a theory can be expressed using a conformal choice that makes the equation governing the metric closely resemble the Einstein equation of general relativity. This is called the \emph{Einstein frame formulation}. The singularity theorems of general relativity can be applied quite directly in this formulation. However, this raises interpretive issues. For example, the constants of nongravitational physics are not constant in inertial frames of this metric \cite{Faraoni}. One can transform back to a metric in whose inertial frames the constants of physics are constant---this is called the \emph{Jordan frame formulation}---but the power of the theorem can be diminished in the process. A more satisfactory approach would be to develop Jordan frame singularity theorems directly.

This has now become possible, thanks to recent developments in the comparison theory of the \emph{Bakry-\'Emery-Ricci} curvature tensor (or simply \emph{Bakry-\'Emery tensor}) (see, e.g., \cite{Lott, WW}). This tensor is defined in terms of the familiar Ricci tensor of a metric $g$ and an additional \emph{weight function} $f$ by
\begin{equation}
\label{eq1.1} {\rm Ric}_f:={\rm Ric} + {\rm Hess} f\ .
\end{equation}
Our particular interest is the case where the metric $g$ is Lorentzian. This was studied in \cite{Case}.

Motivated by the work of \cite{Case} as well as by the cosmology of scalar-tensor gravitation theories, we consider two singularity theorems for Lorentzian metrics with Bakry-\'Emery-Ricci tensor obeying a positivity condition of a form commonly called an energy condition, and with weight function $f$ bounded above. The extrinsic geometry of hypersurfaces has a Bakry-Emery generalization; see section 2 and in particular equation (\ref{eq2.4}) for the definition of Bakry-Emery mean curvature $H_f$. Then we have the following theorems.

\begin{theorem}\label{theorem1.1}
Let $M$ be a spacetime satisfying ${\rm Ric}_f(X,X) \ge 0$ for all timelike vectors
$X$, and suppose that $f \le k$.\footnote
{This is necessary, for consider the Einstein static universe $-dt^2+g(S^{n-1},{\rm can})$ with $f=e^t$. Then a simple calculation yields ${\rm Ric}_f(X,X)\ge e^t>0$, while $H_f=-e^t<0$ for any constant-$t$ hypersurface, yet this spacetime is geodesically complete.}
Let $S$ be a smooth compact spacelike Cauchy surface for $M$ with strictly negative $f$-mean curvature $H_f(S) < 0$. Then every timelike geodesic is future incomplete.
\end{theorem}

\begin{theorem}\label{theorem1.2}
Let $M$ be a spacetime satisfying ${\rm Ric}_f(X,X) \ge -(n-1)$ for all unit timelike vectors $X$, and having smooth compact spacelike Cauchy surface $S$. Suppose that either
\begin{enumerate}
\item[(i)] $f \le k$ and the $f$-mean curvature of $S$ satisfies
\begin{equation}
\label{eq1.2} H_f(S) < -(n-1)e^{\frac{2(k- N)}{(n-1)}}  \,,
\end{equation}
where $N = \inf_{p \in S} f(p)$, or
\item[(ii)] $\nabla f$ is future causal and the $f$-mean curvature of $S$ satisfies $H_f(S) < -(n-1)$.
\end{enumerate}
Then every timelike geodesic is future incomplete.
\end{theorem}

Inequalities appear in the assumptions of these singularity theorems. The inequalities for $H_f$ are open conditions. It is then natural to ask what happens in the borderline cases. The answers are contained in the following rigidity theorems.

\begin{theorem}\label{theorem1.3}
Let $(M,g)$ be a spacetime satisfying ${\rm Ric}_f(X,X) \ge 0$ for all timelike vectors
$X$, and suppose that $f \le k$. Let $S$ be a smooth compact spacelike Cauchy surface for $M$ having $f$-mean curvature $H_f(S) \le 0$. If $M$ is future timelike geodesically complete the future of $S$ splits; i.e., $(J^+(S),g)$ is isometric to $([0,\infty) \times S, -dt^2 \oplus h)$, where $h$ is the induced metric on $S$, and $f$ is independent of $t$.
\end{theorem}

\begin{corollary}\label{corollary1.4}
Let $M$ be a spacetime satisfying ${\rm Ric}_f(X,X) \ge 0$ for all timelike vectors
$X$, and suppose that $|f| \le k$. Suppose further that $M$ admits a constant f-mean curvature spacelike  Cauchy hypersurface $S$. If $(M,g)$ is timelike geodesically complete then $M$ splits along $S$, i.e., $(M,g)$ is isometric
to $((-\infty,\infty) \times S, -dt^2 \oplus h)$,
where $h$ is the induced metric on $S$,
and $f$ is independent of $t$.
\end{corollary}

Corollary \ref{corollary1.4} follows immediately from Theorems \ref{theorem1.1} and \ref{theorem1.3}.  Indeed, one must have $H_f(S) = 0$, otherwise $M$ is incomplete by Theorem \ref{theorem1.1} and its time-dual.  The conclusion then follows from  Theorem~\ref{theorem1.3} and its time-dual.  This corollary answers a question raised by Case \cite[see Conjecture~7.2 thereof and the paragraph following it]{Case}.

\begin{theorem}\label{theorem1.5}
Let $M$ be a spacetime satisfying ${\rm Ric}_f(X,X) \ge -(n-1)$ for all timelike vectors $X$, and suppose that $\nabla f$ is future causal. Let $S$ be a smooth compact spacelike Cauchy surface for $M$ with $f$-mean curvature $H_f(S) \le -(n-1)$.
If the timelike geodesics orthogonal to $S$ are future complete,
$(J^+(S),g)$ is isometric to the warped product $([0,\infty) \times S, -dt^2 \oplus  e^{-2t} h)$, where $h$ is the induced metric on $S$, and $f$ is constant.
\end{theorem}

Section 2.1 contains a discussion of the underlying Bakry-\'Emery modified Raychaudhuri equation. The necessary estimates for this equation are proved in section 2.2. The singularity theorems, Theorems \ref{theorem1.1} and \ref{theorem1.2}, are proved in section 2.3.

Section 3 contains the proofs of the rigidity theorems, Theorems \ref{theorem1.3} and \ref{theorem1.5}. The method of proof involves the use of a local mean curvature flow, modified by a shift of the fixed point in the case of Theorem \ref{theorem1.5}, to show that if the geometry is not sufficiently rigid then the flow will produce a hypersurface to which the singularity theorems apply. The mean curvature flow is briefly discussed in section 3.1. The proofs of the rigidity theorems are given in section 3.2.

While these results are general theorems of Lorentzian geometry, we have already observed that they have application in physics, and this is discussed in section 4. Since initial (big bang) singularities are of primary interest, we begin by giving Theorems \ref{theorem4.1} and \ref{theorem4.2}, which are simply time-reversed versions of the singularity theorems of the introductory section. These two theorems can be applied generally to quite arbitrary scalar-tensor gravitation theories, but we restrict subsequent attention to the Brans-Dicke family in $n=4$ spacetime dimensions with a possible potential function for the scalar field. Theorem \ref{theorem4.6} gives conditions under which the Jordan frame metric of a Brans-Dicke theory will have a big bang type singularity in the past, while \ref{theorem4.7} deals with the case of positive cosmological constant. We then compare these theorems with singularity theorems obtained by a conformal transformation (``Einstein frame'') argument.

Except in section 4, we have no restriction on the spacetime dimension $n\ge 2$.

\medskip
\noindent\emph{Acknowledgements.} The work of EW was supported by a Discovery Grant from the Natural Sciences and Engineering Research Council (NSERC) of Canada. The work of GJG was  supported by NSF grant DMS-1313724 and by a grant from the Simons Foundation (Grant No.  63943). The work of GJG was also supported by NSF grant 0932078 000, while in residence at MSRI, Berkeley, during Fall, 2013.
Both authors wish to express their gratitude to the Park City Math Institute 2013 Summer Research Program, where this work was begun.

\section{Riccati estimates}
\setcounter{equation}{0}

\subsection{The Raychaudhuri equation}

\noindent We will need certain simple estimates governing the behaviour of solutions of the Raychaudhuri equation, the scalar Riccati equation used to study the focusing behavior of timelike geodesic congruences issuing orthogonally from a spacelike hypersurface $\Sigma$. We recall the basic set-up. Let $\gamma$ belong to such a congruence ${\mathcal C}$ and let it be parametrized by proper time $t$, so the geodesics are ``unit speed''. Thus, $g(\gamma',\gamma')=-1$ where $\gamma'=\frac{d}{dt}$. At $\Sigma$ we have $\gamma'\vert_{\Sigma}=\nu$ where $\nu$
is the future directed  unit normal vector field for $\Sigma$. The congruence ${\mathcal C}$ is surface-forming, so for a curve $\gamma\in {\mathcal C}$, we obtain a foliated neighborhood ${\mathcal N}$ in spacetime near $\gamma:[0,T)\to M$ by moving a parameter distance $t<T$ along the congruence from $\Sigma$, provided that $\gamma$ has no focal point to $\Sigma$ in ${\mathcal N}$. These leaves are also spacelike hypersurfaces. The \emph{extrinsic curvature} or \emph{second fundamental form} of the hypersurface $\Sigma_t$ can be defined as
\begin{equation}
\begin{split}
\label{eq2.1} K(t)(X,Y) = -\nu_t\cdot \left ( \nabla_X Y\right )\ ,\ X,Y\in T_{\gamma(t)}\Sigma_t \ ,
\end{split}
\end{equation}
where $\nu_t$ is the future directed  unit normal for $\Sigma_t$. The \emph{expansion scalar} or \emph{mean curvature} of the congruence is
\begin{equation}
\label{eq2.2} H(t):={\rm tr}_h K(t) \ ,
\end{equation}
where $h:=g+\nu\otimes\nu$ is the induced metric on the leaf. We often suppress the argument $t$. Then the Raychaudhuri equation is
\begin{equation}
\label{eq2.3} \frac{\partial H}{\partial t}=-{\rm Ric}(\nu,\nu)-\vert K \vert^2=-{\rm Ric}(\nu,\nu)-|\sigma|^2-\frac{H^2}{(n-1)}\ ,
\end{equation}
where $|K|^2:=h^{ij}h^{kl}K_{ik}K_{jl}$, $\sigma_{ij}:=K_{ij}-\frac{H}{(n-1)}h_{ij}$ is the \emph{shear} (i.e., the tracefree part of $K_{ij}$), and $n$ is the spacetime dimension.

We deal with the \emph{Bakry-\'Emery modified versions}, also simply called the \emph{modified versions} or \emph{$f$-versions} of curvature quantities. The \emph{$f$-mean curvature} is defined along our unit speed timelike geodesic congruence to which $\gamma$ belongs by
\begin{equation}
\label{eq2.4} 
H_f:=
H - \nabla_{\nu} f \equiv H -  f' \ ,
\end{equation}
where we abbreviate $f\circ\gamma$ by simply writing $f$, so that $\frac{df}{dt}:=f'(t):=(f\circ\gamma)'(t)$. The Raychaudhuri equation (\ref{eq2.3}) then yields the inequality
\begin{equation}
\label{eq2.5}
\begin{split}
{H_f}'&\le -{\rm Ric}_f(\gamma',\gamma')-\frac{H^2}{n-1}\\
&\le -{\rm Ric}_f(\gamma',\gamma')-\frac{H_f^2}{n-1}-\frac{2 H_f f'}{n-1}\ .
\end{split}
\end{equation}
It is convenient to normalize $H_f$ using
\begin{equation}
\label{eq2.6} x:=H_f/(n-1)\ .
\end{equation}
Then $x$ is the \emph{normalized $f$-mean curvature} of the leaves of the foliation by $t=const$ hypersurfaces.  The last inequality is then
\begin{equation}
\label{eq2.7} x'\le -\frac{1}{(n-1)}{\rm Ric}_f(\gamma',\gamma')-x^2 -\frac{2xf'}{(n-1)}\ .
\end{equation}

\subsection{The focusing estimates}

\noindent Here we obtain estimates for functions $x$ that obey (\ref{eq2.7}). To obtain the first estimate, note that when ${\rm Ric}_f(\gamma',\gamma')\ge 0$, equation (\ref{eq2.7}) becomes
\begin{equation}
\label{eq2.8} x'\le -x^2 -\frac{2xf'}{(n-1)}\ .
\end{equation}

As discussed above, we are concerned with the congruence ${\mathcal C}$ of future timelike, unit speed geodesics issuing orthogonally from a smooth spacelike hypersurface $\Sigma$. For $\gamma\in {\mathcal C}$ with $p=\gamma(0)\in \Sigma$, we will write $x_p(t):=x\circ\gamma(t)$, so $x_p(t)$ is the normalized $f$-mean curvature of the leaf $\Sigma_t$ at a point reached by traversing $\gamma$ for a proper time $t$ starting from $\gamma(0)=p\in\Sigma$. We similarly write $f_p:=f\circ\gamma$. Then (\ref{eq2.8}) becomes an ordinary differential inequality for $x_p$, and we have the following result:

\begin{lemma}\label{lemma2.1}
Suppose that
\begin{enumerate}
\item[(i)] $x_p$ obeys the inequality (\ref{eq2.8}),
\item[(ii)] there is a $\delta_p>0$ such that $x_p(0)\le -\delta_p$, and
\item[(iii)] $f(q)\le k$ for some $k\in{\mathbb R}$ and all $q\in M$.
\end{enumerate}
Then there exists a $t_p>0$ such that $x_p(t)\to -\infty$ at or before $t_p$, and
$t_p$ depends only on $n$, $k$, $\delta_p$, and $f_p(0)$.
\end{lemma}

\begin{proof}
The proof is a straightforward modification of the proof of \cite[Lemma 3.1]{RW}. (One could also modify the proof of \cite[Proposition 3.3]{Case}.) Fix $p\in\Sigma$ and let $[0,T)$ be the largest such interval on which $x_p(t)$ is defined (possibly $T=\infty$). If $x_p$ has a zero, let $T_0\in (0,T)$ be the first zero; otherwise, $T_0=T$. On $[0,T_0)$, we can divide (\ref{eq2.8}) by $-x_p^2$ and obtain
\begin{equation}
\label{eq2.9}\left ( \frac{e^{-\frac{2f_p}{(n-1)}}}{x_p}\right )'\ge e^{-\frac{2f_p}{(n-1)}}\ .
\end{equation}
Integrating on $[0,t]$, $t<T_0$, we obtain
\begin{equation}
\label{eq2.10}\frac{e^{-\frac{2f_p(t)}{(n-1)}}}{x_p(t)}
-\frac{e^{-\frac{2f_p(0)}{(n-1)}}}{x_p(0)}\ge
\int\limits_0^t e^{-\frac{2f_p(s)}{(n-1)}}ds\ge te^{-\frac{2k}{(n-1)}}\ .
\end{equation}
Solving for $x_p(t)$ and using that $1/\delta_p>-1/x_p(0)$, this yields
\begin{equation}
\label{eq2.11} x_p(t)\le -\frac{e^{\frac{2(k-f_p(t))}{(n-1)}}}{
\frac{1}{\delta_p}e^{\frac{2(k-f_p(0))}{(n-2)}}-t}\ .
\end{equation}
We first observe that the denominator of the quotient in the right-hand side is positive at $t=0$ and decreases linearly, so the denominator cannot diverge to $\infty$. Since the numerator cannot have a zero, then $T_0=T$. Moreover, the denominator is positive for as long as $x_p(t)$ is defined, and so we can write
\begin{equation}
\label{eq2.12} x_p(t)\le -\frac{e^{\frac{2(k-f_p(t))}{(n-1)}}}{\left ( \frac{1}{ \delta_p}e^{\frac{2(k-f_p(0))}{(n-1)}}-t\right )}
\le -\frac{1}{\left ( \frac{1}{\delta_p}e^{\frac{2(k-f_p(0))}{(n-1)}}-t\right )}
\le -\frac{1}{\left ( t_p-t\right )}\ ,
\end{equation}
where
\begin{equation}
\label{eq2.13}t_p:= \frac{1}{\delta_p}e^{\frac{2(k-f_p(0))}{(n-1)}}\ .
\end{equation}
Clearly, $T\le t_p$.
\end{proof}

We now consider the case where ${\rm Ric}_f(X,X)\ge -(n-1)$ for all future timelike unit vectors $X$. In this case, equation (\ref{eq2.7}) becomes
\begin{equation}
\label{eq2.14} x'\le 1-x^2-\frac{2xf'}{(n-1)}\ .
\end{equation}

\begin{lemma}\label{lemma2.2}
As before, let $\gamma\in {\mathcal C}$ belong to the unit speed timelike geodesic congruence issuing from $\Sigma$ orthogonally, with $\gamma(0)=p$. Let $x_p:=x\circ\gamma$ and let $f_p:=f\circ\gamma$. Suppose that
\begin{enumerate}
\item[(i)] $x_p$ obeys the inequality (\ref{eq2.14}),
\item[(ii)] $f_p\le k_p$ for some $k_p\in{\mathbb R}$, and
\item[(iii)] $x_p(0) \le -(1+\delta_p )e^{\frac{2(k_p-f_p(0))}{(n-1)}}$ for some $\delta_p>0$.
\end{enumerate}
Then there exists a $t_p=t_p(\delta_p)>0$ such that $x_p(t)\to -\infty$ at or before $t_p$.
\end{lemma}

\begin{proof}
Along $\gamma$ equation (\ref{eq2.14}) can be written as
\begin{equation}
\label{eq2.15} \left ( e^{\frac{2(f_p(t)-k_p)}{(n-1)}} x_p(t)\right )'
\le e^{\frac{2(f_p(t)-k_p)}{(n-1)}}\left ( 1-x_p^2(t) \right )\ ,
\end{equation}
or
\begin{eqnarray}
\label{eq2.16} y_p'(t)&\le& e^{\frac{2(f_p(t)-k_p)}{(n-1)}}-e^{-\frac{2(f_p-k_p)}{(n-1)}}y_p^2(t)\ ,\\
\label{eq2.17} y_p(t)&:=& e^{\frac{2(f_p(t)-k_p)}{(n-1)}}x_p(t)\ .
\end{eqnarray}
Since $f_p(t)\le k_p$, then (\ref{eq2.16}) yields
\begin{equation}
\label{eq2.18} y_p'\le 1-y_p^2\ .
\end{equation}

Similarly to the previous proof, let $[0,T)$ be the largest interval on which $y_p$ is defined (possibly $T=\infty$) and let $T_0$ be the first point at which $y_p=-1$; if there is no such point, then let $T_0=T$. By assumption (iii) we have $y_p(0)\le -(1+\delta_p)<-1$, so $1-y_p^2<0$ on $[0,T_0)$. Then on $[0,T_0)$ equation (\ref{eq2.18}) yields
\begin{equation}
\label{eq2.19} \frac{y_p'}{y_p^2-1}\le -1\ ,
\end{equation}
so
\begin{eqnarray}
\label{eq2.20} y_p(t)&\le & -{\rm\, coth\,}(t_p-t) \ , \\
\label{eq2.21} t_p&:=&  {\rm arctanh\,}\frac{1}{1+\delta_p}\ .
\end{eqnarray}
Thus $T_0=T$ and (\ref{eq2.17}) and (\ref{eq2.20}) imply that
\begin{equation}
\label{eq2.22} x_p(t)\le -e^{\frac{2(k_p-f_p(t))}{(n-1)}}{\rm coth\,}(t_p-t) \ .
\end{equation}
Therefore $x_p(t)\to -\infty$ as $t\nearrow T$ for some $0<T\le t_p$.
\end{proof}

The time $t_p=t_p(\delta_p)$ depends on $k_p$ only indirectly, in that $k_p$ determines $\delta_p$ in condition (iii) of the theorem. When $k_p$ is realized at $p$, condition (iii) simplifies, as occurs in the following result:

\begin{corollary}\label{corollary2.3} Let $\gamma$ be as in Lemma \ref{lemma2.2}. Say that
\begin{enumerate}
\item[(i)] $x_p$ obeys the inequality (\ref{eq2.14}),
\item[(ii)] $\nabla f$ is future causal, and
\item[(iii)] $x_p(0)\le -(1+\delta_p)$ for some $\delta_p>0$.
\end{enumerate}
Then there is a $t_p=t_p(\delta_p)>0$ such that $x_p(t)\to -\infty$ at or before $t_p$.
\end{corollary}

\begin{proof}
Condition (i) is the same as condition (i) of Lemma \ref{lemma2.2}. Since $\nabla f$ is future causal and $\gamma\in {\mathcal C}$ is future timelike, then $\nabla_{\gamma'}f\le 0$. Therefore, $f_p$ is (at least weakly) decreasing along $\gamma$, so $f_p(t) \le f_p(0)=: k_p$ in the terminology of Lemma \ref{lemma2.2}. Then condition (ii) of Lemma \ref{lemma2.2} holds for $p$, and condition (iii) of that lemma reduces to condition (iii) of this corollary.
\end{proof}

\subsection{Proofs of singularity theorems}
We are now in a position to prove Theorems \ref{theorem1.1} and \ref{theorem1.2} from the Introduction. Each theorem rests on the following standard result.

\begin{lemma}\label{lemma2.4}
Suppose that $S$ is a spacelike Cauchy surface and $\sigma$ is a future complete timelike geodesic. Then there is an arbitrarily long future timelike geodesic $\gamma$ leaving $S$ orthogonally and having no focal point to $S$.
\end{lemma}

\begin{proof}
Let $q$ along $\sigma$ lie at Lorentzian distance $d(S,q)=t_q$ from $S$. Since $\sigma$ is future complete, we can choose $q$ so that $t_q$ is arbitrarily large. Let $\gamma : [0, t_q] \to M$ be a unit speed maximal timelike geodesic segment from $S$ to $q$. Then $\gamma$ must leave $S$ orthogonally. The length of (that is, the proper time along) $\gamma$ is $L(\gamma) = d(S,q)=t_q$, where $d(\cdot , \cdot)$ denotes the Lorentzian distance. Since the length of $\gamma$ realizes the distance from $S$ to $q$, it cannot have a focal point to $S$ before $q$.
\end{proof}

\begin{proof}[Proof of Theorem \ref{theorem1.1}] Since ${\rm Ric}_f(X,X)\ge 0$ for all timelike vectors $X$, the inequality (\ref{eq2.8}) governs $x$. By compactness of $S$, the restriction of $f$ to $S$ has a lower bound $N$, so $f_p(0)\ge N$ in Lemma \ref{lemma2.1}. Also by compactness, and since $H_f<0$ with $S$ and $f$ smooth, there exists $\delta > 0$ such that $H_f(S) \le -(n-1)\delta$ and so $x(0)\le -\delta$; i.e., $x_p(0)\le -\delta$ for all $p\in S$, so we can take $\delta_p=\delta$ independent of $p\in S$ in Lemma \ref{lemma2.1}. By assumption, $f\le k$ everywhere, so all conditions for Lemma \ref{lemma2.1} are satisfied, with $S \equiv \Sigma$. Furthermore, they are satisfied for every $p\in S$ with the same $\delta_p=\delta$. Then Lemma \ref{lemma2.1} implies that every future timelike geodesic leaving $S$ orthogonally must have a focal point for $t\le t_p\le t_0$, where $t_0:=\sup_{p\in S} t_p(0)$. But if there were a future complete timelike geodesic, this would contradict Lemma \ref{lemma2.4}.
\end{proof}

\begin{proof}[Proof of Theorem \ref{theorem1.2} part (i).]
Since ${\rm Ric}_f(X,X)\ge -(n-1)$ for all timelike vectors $X$, the inequality (\ref{eq2.14}) now governs $x$. From compactness of $S$, we have $\inf_{p\in S} f(p) = N\in {\mathbb R}$.

From equation (\ref{eq1.2}) we have for $p\in\Sigma$ that
\begin{equation}
\label{eq2.23}
\begin{split}
&\,\, x_p(0)<-e^{\frac{2(k-N)}{(n-1)}}\\
\Rightarrow &\,\, x_p(0)<-e^{\frac{2(k_p-f_p(0))}{(n-1)}}\\
\Rightarrow &\,\, x_p(0)\le -(1+\delta) e^{\frac{2(k_p-f_p(0))}{(n-1)}}\\
\end{split}
\end{equation}
for some $\delta>0$, where in the last step we invoke the compactness of $S$. Then by Lemma \ref{lemma2.2} we have that $x_p(t)\to -\infty$ at or before $t=t_p(\delta) = {\rm arctanh\,}\frac{1}{1+\delta} =:t_0(\delta)$. That is, $t_p$ is uniformly bounded above by $t_0=t_0(\delta)$. Again, if there were a future complete timelike geodesic, this would contradict Lemma \ref{lemma2.4}.
\end{proof}

\begin{proof}[Proof of Theorem \ref{theorem1.2} part (ii).]
Again the inequality (\ref{eq2.14}) governs $x$. Since $H_f(S)<-(n-1)$ with $S$ compact and $f$ and $S$ smooth, then $H_f(S)\le -(n-1)\delta$ for some $\delta>0$, and so $x_p(0)\le -(1+\delta)$ uniformly for any $p\in S$. Then the conditions of Corollary \ref{corollary2.3} are fulfilled with $S\equiv \Sigma$, and so there is a uniform bound $t_0=t_0(\delta)>0$ such that, for any $p\in S$, $x_p(t)\to -\infty$ for some $t \le t_0(\delta)$. And again, if there were a future complete timelike geodesic, this would contradict Lemma \ref{lemma2.4}.
\end{proof}

Finally, we note that a weakening of the causality assumptions used in the above theorems will not necessarily eliminate the problem of geodesic incompleteness, or at least not entirely. The following result dispenses with the assumption of global hyperbolicity, but shows that there will still exist future incomplete timelike geodesics, though in this case we can no longer establish that all timelike geodesics are future incomplete (cf \cite[Theorem 4, p 272]{HE}).

\begin{theorem}\label{theorem2.5}
Let $M$ be a spacetime satisfying ${\rm Ric}_f(X,X) \ge 0$ for all timelike vectors
$X$, and suppose that $f \le k$. Let $S$ be a smooth compact acausal spacelike hypersurface with strictly negative $f$-mean curvature, $H_f(S) < 0$. Then there exists an inextendible future incomplete timelike geodesic in $M$.
\end{theorem}

\begin{proof} If $S$ is a spacelike {\it future Cauchy surface}, i.e., if its future Cauchy horizon $H^+(S) = \emptyset$, then the proof of Theorem \ref{theorem1.1} goes through verbatim. So suppose $H^+(S) \ne  \emptyset$.  In this case one can apply for example \cite[Main Lemma]{Galloway} to obtain an inextendible future timelike geodesic in $D^+(S)$ with initial endpoint $p\in S$ (and initial direction orthogonal to $S$) which maximizes distance from each of its points to $S$. By Lemma \ref{lemma2.1}, this geodesic cannot maximize distance beyond length $t_p$, and so must be incomplete.
\end{proof}

\begin{remark}\label{remark2.6}
Likewise, if the global hyperbolicity assumption in Theorem \ref{theorem1.2} is replaced by the assumption that $S$ is a smooth compact acausal spacelike hypersurface, then we may similarly conclude in that case that some (but not necessarily every) timelike geodesic is future incomplete.
\end{remark}

\section{Rigidity}
\setcounter{equation}{0}

\subsection{Extrinsic curvature flows}

\noindent The proofs of the rigidity theorems below use a modified version of the mean curvature flow of a hypersurface to construct a small pointwise deformation of the hypersurface. We therefore first recall some standard facts about extrinsic curvature flows for hypersurfaces in a Lorentzian manifold. For a detailed treatment, see \cite[Chapter 2]{Gerhardt}. Consider a family of embeddings $F:[0,\varepsilon)\times\Sigma\to M : s\mapsto F(s,\cdot)$. For each $s$, this embeds $\Sigma$ (which we identify with $\Sigma_0$) as a spacelike hypersurface $\Sigma_s$ in $(M,g)$ such that
\begin{equation}
\label{eq3.1}
\begin{split}
\frac{\partial F}{\partial s}=&\,\, \phi\nu\ ,\\
F(0,\cdot)=&\,\, {\rm id}\ ,
\end{split}
\end{equation}
where $\phi$ depends on the mean curvature of $F(s,\cdot)$ in $M$ and $\nu$ is the corresponding timelike unit normal field. When $\phi=H$, a solution of (\ref{eq3.1}) is called a \emph{mean curvature flow}. However, we are interested here in the case of
\begin{equation}
\label{eq3.2} \phi=H_f-c=H-\nabla_{\nu}f-c\ ,
\end{equation}
where $c$ is a constant. Then we call such a solution a \emph{$(c,f)$-mean curvature flow}. Thus fixed points of the flow are hypersurfaces of constant $f$-mean curvature $H_f=c$. Under this flow, $\phi$ evolves as\footnote
{For the derivation, see \cite[Lemma 2.3.4]{Gerhardt}. In the notation of \cite{Gerhardt}, set $\sigma=-1$ and $\Phi={\rm id}$. As well, $F$ as used in \cite{Gerhardt} is our $H$, $f$ as used in \cite{Gerhardt} is our $-c-\nabla_{\nu}f$, and $F_{ij}$ as used in \cite{Gerhardt} is our $h_{ij}$ (the induced metric on $\Sigma_s$). It is then necessary to re-write a $\nabla_{\nu}\nabla_{\nu}f$ term using that $\nabla_{\nu}\nu=\nabla\log \phi$ in our case.}
\begin{equation}
\label{eq3.3} \frac{\partial \phi}{\partial s}=\Delta_{\Sigma_s}\phi-D_{\Sigma_s} f \cdot D_{\Sigma_s} \phi -\left (|K|_{h_s}^2+{\rm Ric}_f(\nu,\nu) \right )\phi\ .
\end{equation}
where $\Delta_{\Sigma_s}\phi := D_{\Sigma_s}\cdot D_{\Sigma_s}\phi$ is the Laplacian (the trace of the Hessian formed from the Levi-Civita connection $D_{\Sigma_s}$ of the induced metric $h_{ij}(s)$) of $\phi$ on $\Sigma_s:=(\Sigma,h_{ij}(s))$ and $D_{\Sigma_s} f \cdot D_{\Sigma_s} \phi =h(s)(D_{\Sigma_s} f, D_{\Sigma_s} \phi)$, but ${\rm Ric}_f$ is the Bakry-\'Emery tensor of the ambient spacetime.

\begin{lemma}\label{lemma3.1}
Let $(\Sigma,h_{ij}^0)\hookrightarrow (M,g)$ be a closed spacelike hypersurface such that $\phi:=H_f-c\le 0$ for all $p\in\Sigma$. There is an $\varepsilon>0$ such that the $(c,f)$-mean curvature flow $F:[0,\varepsilon)\times\Sigma\to (M,g)$ obeying (\ref{eq3.1}, \ref{eq3.2}) exists. Furthermore, $\varepsilon$ can be chosen so that either $\phi(t,q)<0$ for all $t\in(0,\varepsilon)$ and all $q\in\Sigma$ or $\phi\equiv 0$ for all $t\in[0,\varepsilon)$ and all $q\in\Sigma$. In particular, if in addition $\phi(0,p)<0$ for some $p\in\Sigma$, then $\phi(t,q)<0$ for all $t\in(0,\varepsilon)$ and all $q\in\Sigma$.
\end{lemma}

\begin{proof}
For $\Sigma$ a closed spacelike hypersurface, \cite[Theorem 2.5.19]{Gerhardt}
guarantees a smooth solution of (\ref{eq3.1}, \ref{eq3.2}) on $[0,\varepsilon)\times\Sigma$ for some $\varepsilon>0$. If $\phi(t,q)<0$ for all $(t,q)\in (0,\epsilon)\times\Sigma$, we are done, so say $\phi(\delta,q)\ge 0$ for some $\delta\in (0,\epsilon)$ and some $q\in\Sigma$. Since $\phi(0,\cdot)\le 0$, this implies that the maximum $M$ of $\phi$ on $[0,\delta]\times\Sigma$ is achieved for some $t>0$. By the strong maximum principle \cite[Theorem 2.7] {Lieberman}, this can only happen if $\phi\equiv 0$ on $[0,\delta]\times\Sigma$, which can only occur if $\phi(0,\cdot)\equiv 0$.
\end{proof}

\subsection{Proofs of the rigidity theorems}

\begin{proof}[Proof of Theorem \ref{theorem1.3}.]
Introduce Gaussian normal coordinates in a neighborhood $U$ of $S$ in $J^+(S)$,
\begin{equation}
\label{eq3.4} g = -dt^2 + h_{ij}dx^idx^j\ ,\ t \in [0,\epsilon) \ .
\end{equation}
Let $H = H(t)$ be the mean curvature of the slice $S_t = \{t \} \times S$, with $H_f(t)=H(t)-f'(t)$ and $x(t)=H_f(t)/(n-1)$ as usual.

Because ${\rm Ric}(X,X)\ge 0$ for all timelike vectors $X$, $x$ obeys (\ref{eq2.8}), with $x(0)\le 0$ by assumption. Therefore $x(t)$ obeys
\begin{equation}
\label{eq3.5} x'+\frac{2f'}{(n-1)}x\le -x^2\ ,\ x(0)\le 0\ .
\end{equation}
Multiplying by $e^{2f/(n-1)}$ and integrating to the future along the $t$-geodesics yields
\begin{equation}
\label{eq3.6} e^{\frac{2f(t)}{(n-1)}}x(t)-e^{\frac{2f(0)}{(n-1)}}x(0)=-\int\limits_0^t e^{\frac{2f(u)}{(n-1)}}x^2(u)du\le 0\ .
\end{equation}
Using $x(0)\le 0$, we obtain that $x(t)\le 0$ and thus $H_f(t)\le 0$ for all $t\ge 0$ in the coordinate domain.

Suppose for some $t_0$, $H_f(t_0)$ is strictly less than zero at some point. If $H_f(t_0)<0$ everywhere, then by Theorem \ref{theorem1.1} every timelike geodesic will be future incomplete, contrary to assumption. If, however, there are both points where $H_f(t_0)=0$ and points where $H_f(t_0)<0$, then let the hypersurface and its induced metric be initial data for a $(c,f)$-mean curvature flow (\ref{eq3.1}, \ref{eq3.2}) on an interval $s\in [0,\varepsilon)$ with $c=0$. Then $\phi(0)\le 0$ with $\phi(0)=0$ at some points and $\phi(0)<0$ at others. By Lemma \ref{lemma3.1}, such a flow always exists for $\varepsilon>0$ small enough, and for $s>0$, $\phi$ will be strictly less than zero. Then by (\ref{eq3.2}) with $c=0$, the $f$-mean curvature of the deformed hypersurfaces will be strictly less than zero. Furthermore, the deformed hypersurfaces are spacelike Cauchy surfaces.
Using any of these hypersurfaces as the hypersurface $S$ in the assumptions of Theorem \ref{theorem1.1}, then that theorem implies that every timelike geodesic will be future incomplete, contrary to assumption.

Thus, we have $H_f(t) = 0$ for all $t \in [0,\epsilon)$. Inequality (\ref{eq2.5}) then implies that $H=0$ and, therefore by \eqref{eq2.4}, $\frac{\partial f}{\partial t} = 0$ on $U$. The Raychaudhuri equation (\ref{eq2.3}) then implies that the second fundamental form $K$ vanishes identically, so each $S_t$ is totally geodesic. Solving $0=K:=\frac12 \frac{\partial}{\partial t}h$  gives $h_{ij}(t) = h_{ij}(0)$ and we obtain the desired splitting on $U$. Since the normal geodesics are future complete, this splitting can be continued indefinitely.
\end{proof}

\begin{proof}[Proof of Theorem \ref{theorem1.5}.] Introduce Gaussian normal coordinates in a neighborhood $U$ of $S$ in $J^+(S)$ and define the slices $S_t = \{t \} \times S$ as above.
The normalized $f$-mean curvature $x(t):=H_f(t)/(n-1)$ of $S_t$ satisfies (\ref{eq2.14}) with $x(0) \le -1$. Observe that if we choose $\epsilon$ sufficiently small so that $H_f(t) < 0$ for all $t\in [0,\epsilon)$, then $xf' \ge 0$ for $t\in [0,\epsilon)$, since by assumption $f'\equiv\gamma'\cdot \nabla f \le 0$. Then (\ref{eq2.14}) reduces to $x'\le 1-x^2$ with $x(0) \le -1$. Elementary comparison with the solution to $y' = 1 -y^2$, $y(0) = -1$, implies that $x(t) \le -1$ for all $t\in [0,\epsilon)$. Hence, we have that $H_f(t) \le -(n-1)$ for all $t \in [0,\epsilon)$.

Suppose that, for some $t_0$, $H_f(t_0)$ is strictly less than $-(n-1)$ at some point. Then, as in the proof of Theorem \ref{theorem1.3}, we can employ a $(c,f)$-mean curvature flow (\ref{eq3.1}, \ref{eq3.2}), this time with $c=-(n-1)$, and invoke Lemma \ref{lemma3.1} to obtain a nearby spacelike Cauchy surface with $f$-mean curvature $H_f<-(n-1)$ pointwise. Using this hypersurface as the hypersurface $S$ in the assumptions of Theorem \ref{theorem1.2}, then that theorem again implies that every timelike geodesic will be future incomplete, contrary to assumption.

Thus, we have $H_f(t) = -(n-1)$ for all $t \in [0,\epsilon)$. Inequality (\ref{eq2.5}) then implies that
$\frac{\partial f}{\partial t} = 0$ on $U$, and $f$-mean curvature reduces to ordinary mean curvature.
Hence, $H(t) = -(n-1)$, and the Raychaudhuri equation (\ref{eq2.3}) implies that the shear $\sigma$ vanishes identically. It follows that each $S_t$ is umbilic, with second fundamental form $K_{ij} = -\frac12 \frac{\partial h_{ij}}{\partial t} = h_{ij}$. Solving for $h_{ij}$ gives $h_{ij}(t) = e^{-2t}h_{ij}(0)$ and we obtain the desired warped product splitting on $U$. Since the normal geodesics are future complete, this splitting can be continued indefinitely.
\end{proof}

\section{Application: Scalar-tensor cosmology}
\setcounter{equation}{0}

\noindent In cosmology, so-called \emph{big bang} singularities are often of primary interest. In a spacetime with a big bang, every timelike geodesic is \emph{past}-incomplete. Big bang singularities can be addressed using time-reversed versions of the theorems in the Introduction. In what follows, we continue to define the $f$-mean curvature of a spacelike hypersurface with respect to the future direction.

 \begin{theorem}\label{theorem4.1}
Let $M$ be a spacetime satisfying ${\rm Ric}_f(X,X) \ge 0$ for all timelike vectors
$X$, and suppose that $f \le k$. Let $S$ be a smooth compact spacelike Cauchy surface for $M$ with strictly positive $f$-mean curvature, $H_f(S) > 0$. Then every timelike geodesic is past incomplete.
\end{theorem}

\begin{theorem}\label{theorem4.2}
Let $M$ be a spacetime satisfying ${\rm Ric}_f(X,X) \ge -(n-1)$ for all unit timelike vectors $X$, and having smooth compact spacelike Cauchy surface $S$. Suppose that either
\begin{enumerate}
\item[(i)] $f \le k$ and the $f$-mean curvature of $S$ satisfies
\begin{equation}
\label{eq4.1} H_f(S) > (n-1)e^{\frac{2(k- N)}{(n-1)}}  \,,
\end{equation}
where $N = \inf_{p \in S} f(p)$, or
\item[(ii)] $\nabla f$ is past causal and the $f$-mean curvature of $S$ satisfies $H_f(S) > (n-1)$.
\end{enumerate}
Then every timelike geodesic is past incomplete.
\end{theorem}

\begin{remark}\label{remark4.3}
The proofs of these theorems are obvious. Theorems \ref{theorem1.3}, \ref{theorem1.5}, and \ref{theorem2.5} and Remark \ref{remark2.6} also have obvious time-reversed versions.
\end{remark}

Scalar fields are routinely invoked in modern cosmology. Scalar fields arise in the cosmological inflation scenario, string cosmology, and models that attempt to explain the observed accelerating expansion of the present universe. Cosmological inflation is usually described by scalar fields coupled to Einstein gravity, but string models use a Brans-Dicke type dilaton coupling, while models of the observed acceleration in the cosmological expansion rely on a variety of scenarios including scalar-tensor gravitation.

The primary examples of scalar-tensor gravitation theories are the Brans-Dicke family of theories. In $n=4$ spacetime dimensions, this family is parametrized by a number $\omega\in \left ( -\frac32,\infty\right )$. In addition to the spacetime metric, the theory also contains a scalar field $\varphi>0$. In the conformal gauge usually called the \emph{Jordan frame}, the equations of Brans-Dicke theory are given by the system \cite[p 9]{Faraoni}
\begin{eqnarray}
\label{eq4.2} {\rm Ric}-\frac12 R g+\frac{1}{2\varphi}V(\varphi) g
&=&\frac{1}{\varphi}\left ( {\rm Hess\,}\varphi
-g\square \varphi \right )+ \frac{8\pi}{\varphi}T\nonumber \\
&& +\frac{\omega}{\varphi^2} \left ( \nabla \varphi \otimes \nabla \varphi
-\frac12 g \vert \nabla \varphi \vert ^2\right )\ ,\\
\label{eq4.3} \square \varphi &=&\frac{1}{3+2\omega}\left ( 8\pi {\rm tr}_gT +\varphi V'(\varphi) -2V(\varphi)\right )\ .
\end{eqnarray}
Here $\square:=\nabla^i\nabla_i$ is the d'Alembert operator (or spacetime scalar Laplacian), $R:={\rm tr}_g{\rm Ric}$ is the scalar curvature, and $T$ is the stress-energy-momentum tensor of nongravitational matter, and does not depend on $\varphi$ (this is called \emph{minimal coupling}). Note that the notation $\vert d\varphi \vert^2$ is shorthand for the \emph{Lorentzian} norm, so $\vert d\varphi \vert^2<0$  when $d\varphi$ is timelike.

The function $V(\varphi)$ was not present in the original formulation \cite{BD}, but is used in cosmological models. It is typically taken to be a homogeneous function of $\varphi$, often a polynomial. Define
\begin{eqnarray}
\label{eq4.4} f&:=&-\log\varphi\ ,\\
\label{eq4.5} W(f)&:=&-\frac16 e^f V\left ( e^{-f}\right )\ ,\\
\label{eq4.6} \Lambda(f)&:=&\frac{6(1+\omega)W(f)-3W'(f)}{(3+2\omega)}=-\frac{1}{2(3+2\omega)} \left [ V'(\varphi)+(1+2\omega)\frac{1}{\varphi}V(\varphi)\right ]\ .
\end{eqnarray}
Then we can re-write the above equations in the form
\begin{eqnarray}
\label{eq4.7} {\rm Ric}_f +\Lambda(f)g&=& 8\pi e^f \left [ T-\left ( \frac{1+\omega}{3+2\omega} \right ) g {\rm tr}_g T \right ] +(1+\omega)df \otimes df \\
\label{eq4.8} \square f -\vert df \vert^2 &=& -\frac{2}{3+2\omega} \left ( 3W'(f)+3W(f)
+4\pi e^f{\rm tr}_g T \right ) \ ,
\end{eqnarray}
using that $\varphi=:e^{-f}$.

Consider the special case of $V(\varphi)=-\left ( \frac{3+2\omega}{1+\omega} \right )\lambda \varphi$ for some constant $\lambda$. Then $W=\left (\frac{3+2\omega}{1+\omega}  \right )\frac{\lambda}{6}$ and $\Lambda(f)=\lambda$. In other words, a linear potential yields a cosmological constant. (From (\ref{eq4.3}), we also see that a linear potential with $\frac{\lambda}{1+\omega}>0$ gives $\varphi$ a mass.)

The following two simple lemmata translate between inequalities expressed in a form natural to Brans-Dicke theory and the Bakry-\'Emery form used in the assumptions of our theorems. Since the Brans-Dicke theory is posed in $n=4$ dimensions (though it can be formulated for $n\ge 3$), we restrict consideration to $n=4$ from here onward.

\begin{lemma} \label{lemma4.4}
Let $X$ be an arbitrary unit timelike vector $g(X,X)=-1$. Assume that $\omega \ge -1$ and that the $\omega$-energy condition (see, e.g., \cite{Woolgar}) holds, so that
\begin{equation}
\label{eq4.9} T(X,X) +\left (\frac{1+\omega}{3+2\omega}\right )\left ( {\rm tr}_gT\right ) \ge 0\ .
\end{equation}
\begin{enumerate}
\item[(a)] If $V'(\varphi)+(1+2\omega)\frac{1}{\varphi}V(\varphi)\le 0$ then ${\rm Ric}_f(X,X)\ge 0$.
\item[(b)] If $V'(\varphi)+(1+2\omega)\frac{1}{\varphi}V(\varphi)\le 6(3+2\omega)$ then ${\rm Ric}_f(X,X)\ge -3$.
\end{enumerate}
\end{lemma}

\begin{proof} Consider (\ref{eq4.7}) termwise, using $g(X,X)=-1$ and (\ref{eq4.6}).\end{proof}

\begin{lemma}\label{lemma4.5}
\begin{eqnarray}
\label{eq4.10} H_f>0 &\Longleftrightarrow& H>-\frac{1}{\varphi}\nabla_{\nu}\varphi\ ,\\
\label{eq4.11} H_f(S) > 3e^{\frac23(k- N)} &\Longleftrightarrow& H>-\frac{1}{\varphi}\nabla_{\nu}\varphi + 3\left ( \varphi_1/\varphi_0\right )^{2/3}\ , \\
\label{eq4.12} H_f(S) > 3 &\Longleftrightarrow& H>-\frac{1}{\varphi}\nabla_{\nu}\varphi +3\ ,
\end{eqnarray}
where $k:=\sup_{J^-(S)}f$, $\varphi_0:=\inf_{J^-(S)}\varphi$, and $\varphi_1:=\sup_S\varphi$.
\end{lemma}

\begin{proof}
Use $f=-\log \varphi$. We note that the left-hand expression in (\ref{eq4.11}) is equation (\ref{eq1.2}) with $n=4$ and with the sense of time reversed. The right-hand expression is simply what one obtains by replacing $k$ and $N$, which are defined in terms of $f$, by $\varphi_0$ and $\varphi_1$, which are defined in terms of $\varphi$.
\end{proof}

Then the following theorems give conditions under which Brans-Dicke theory must have a big bang singularity.

\begin{theorem}\label{theorem4.6}
Let $(M,g,\varphi)$ be a spacetime governed by equations (\ref{eq4.2}, \ref{eq4.3}) for some fixed $\omega\ge -1$. Assume that \begin{enumerate}
\item[(a)] $T$ obeys the $\omega$-energy condition (\ref{eq4.9}) for all timelike vectors $X$,
\item[(b)] $\varphi \ge \varphi_0>0$ for some $\varphi_0\in {\mathbb R}^+$,
\item[(c)] $V'(\varphi)+(1+2\omega)\frac{1}{\varphi}V(\varphi)\le 0$, and
\item[(d)] there is a smooth compact Cauchy surface $S$ for $M$ with mean curvature $H$ obeying
\begin{equation}
\label{eq4.13} H(S) > -\frac{1}{\varphi}\nabla_{\nu}\varphi \ .
\end{equation}
\end{enumerate}
Then every timelike geodesic is past incomplete.
\end{theorem}

\begin{proof}
Using Lemmata \ref{lemma4.4} and \ref{lemma4.5} it is easily seen that our assumptions verify the assumptions of Theorem \ref{theorem4.1}, which we then invoke.
\end{proof}

\begin{theorem}\label{theorem4.7}
Let $(M,g,\varphi)$ be a spacetime governed by equations (\ref{eq4.2}, \ref{eq4.3}) with $\omega\ge -1$. Assume that \begin{enumerate}
\item[(a)] $T$ obeys the $\omega$-energy condition (\ref{eq4.9}) for all timelike vectors $X$ and
\item[(b)] $V'(\varphi)+(1+2\omega)\frac{1}{\varphi}V(\varphi)\le 6(3+2\omega)$, and
\end{enumerate}
Further assume that either
\begin{enumerate}
\item[(c.i)] there is a smooth compact Cauchy surface $S$ for $M$ with $\varphi_0:=\inf_{J^-(S)}\varphi>0$ and
\item[(d.i)] mean curvature $H$ of $S$ obeys
\begin{equation}
\label{eq4.14} H(S) > 3\left ( \varphi_1/\varphi_0\right )^{2/3} -\frac{1}{\varphi}\nabla_{\nu}\varphi \ ,
\end{equation}
where $\varphi_1 = \sup_{p \in S} \varphi(p)$
\end{enumerate}
or
\begin{enumerate}
\item[(c.ii)] $\nabla \varphi$ is future causal and
\item[(d.ii)] there is a smooth compact Cauchy surface $S$ for $M$ with mean curvature $H$ obeying
\begin{equation}
\label{eq4.15} H(S) > 3-\frac{1}{\varphi}\nabla_{\nu}\varphi \ .
\end{equation}
\end{enumerate}
Then every timelike geodesic is past incomplete.
\end{theorem}

\begin{proof}
The proof consists in using Lemmata \ref{lemma4.4} and \ref{lemma4.5} to verify the assumptions of Theorem \ref{theorem4.2}. We note that $\nabla\varphi$ is future causal iff $\nabla f$ is  past causal.
\end{proof}

Brans-Dicke theories admit a so-called \emph{Einstein frame} formulation, meaning that solutions of the theory are described by a metric ${\tilde g}$ related to the Jordan frame metric $g$ by
\begin{equation}
\label{eq4.16} {\tilde g}=\varphi g\ .
\end{equation}
Equation (\ref{eq4.2}) can then be written in the form
\begin{equation}
\label{eq4.17}{\widetilde {\rm Ric}}=\frac{8\pi}{\varphi}
\left [ T-\frac12 {\tilde g} {\rm tr}_{\tilde g}T\right ]
+\frac{(3+2\omega)}{2\varphi^2}d \varphi \otimes d \varphi
+\frac{1}{2\varphi^2}V(\varphi){\tilde g}\ ,
\end{equation}
where ${\widetilde {\rm Ric}}$ is the Ricci tensor of ${\tilde g}$. Except for the $\frac{1}{\varphi}$ multiplying $T$, this equation has the form of Einstein's equation for general relativity with a scalar field $\log \varphi$.

Then it is reasonable to ask whether Theorems \ref{theorem4.6} and \ref{theorem4.7} follow from standard theorems applied to the Einstein frame formulation, or whether they are genuinely distinct, new results. To address this question, we focus only on the incompleteness statement in Theorem \ref{theorem4.6}. Similar considerations will apply to the splitting result and to Theorem \ref{theorem4.7}. We have the following analogue of Theorem \ref{theorem4.6}, whose proof follows immediately from standard results that make no use of the Bakry-\'Emery tensor.

\begin{theorem}\label{theorem4.8}
Let $(M,g,\varphi)$ be a spacetime governed by equation (\ref{eq4.17}) for any fixed $\omega> -3/2$ and $\varphi>0$. Assume that \begin{enumerate}
\item[(a)] $T$ obeys the \emph{strong energy condition} $T(X,X)-\frac12{\tilde g}(X,X) tr_{\tilde g}T\ge 0$ for all timelike vectors $X$,
\item[(b)] $V(\varphi)\le 0$, and
\item[(c)] there is a smooth compact Cauchy surface $S$ for $M$ with mean curvature $H$ obeying
\begin{equation}
\label{eq4.18} H(S) > 0 \ .
\end{equation}
\end{enumerate}
Then every timelike ${\tilde g}$-geodesic is past incomplete. Furthermore, if $\varphi\ge \varphi_0>0$, then every past timelike geodesic in $(M,g)$ is incomplete.
\end{theorem}

\begin{proof}
Our assumptions imply that ${\widetilde {\rm Ric}}(X,X)\ge 0$ for all timelike $X$. Now invoke Theorem \ref{theorem4.1} with $f=0$ to prove that every past timelike ${\tilde g}$-geodesic is incomplete.

To prove that every past timelike $g$-geodesic is incomplete, consider the function $\psi$ that maps each point $q\in S$ to the ${\tilde g}$-length of the past inextendible timelike ${\tilde g}$-geodesic that begins (in the past directed sense of course) at $q\in S$ and meets $S$ orthogonally there. As none of these geodesics is ${\tilde g}$-complete, $\psi$ is finite-valued for each $p\in S$. Indeed, Lemma \ref{lemma2.1} can be applied here with $f$ set to zero (thus $k_p=f_p=0$ in the statement of that lemma), and shows that $\psi(p)\le t_p=1/\delta_p$ (see (\ref{eq2.13})). Then because $S$ is compact and smooth, $\delta_p\ge \delta>0$ on $S$ so $\psi\le 1/\delta$ on $S$.

But assume by way of contradiction that there is a past complete $g$-geodesic $\gamma:[0,\infty)\to M$ (we choose the parameter to increase to the past). Clearly this curve has infinite length as measured by $g$-proper time. Then by (\ref{eq4.16}) and assumption (b), it has infinite length as measured by
${\tilde g}$ as well. We may assume that $\gamma(0)\in S$. Choose an unbounded sequence $t_k>t_{k-1}$, so the points $p_k=\gamma(t_k)$ recede into the past along $\gamma$. Consider the ${\tilde g}$-maximal timelike curves $\zeta_k$ joining the $p_k$ to $S$. Then each $\zeta_k$ is a timelike ${\tilde g}$-geodesic which meets $S$ orthogonally and, being ${\tilde g}$-maximal, has ${\tilde g}$-length greater than or equal to the
${\tilde g}$-length of $\gamma:[0,t_k]\to M$.
But we just established that the ${\tilde g}$-length of $\gamma$ is unbounded, so the $\zeta_k$ form a sequence of past timelike ${\tilde g}$-geodesics that begin on $S$, are orthogonal to $S$, and have unbounded ${\tilde g}$-length, which contradicts that $\psi$ is bounded.
\end{proof}

Then is it possible that the $g$-incompleteness statement in Theorem \ref{theorem4.6} is merely a consequence of Theorem \ref{theorem4.8}? There are in fact differences in the assumptions and applicability of the two theorems. The main distinction between these theorems lies in the energy condition imposed on non-gravitational matter. Clearly
\begin{equation}
\label{eq4.19} T(X,X)-\frac12{\tilde g}(X,X) tr_{\tilde g}T=T(X,X)-\frac12 g(X,X) tr_gT\ ,
\end{equation}
so the strong energy condition holds on $(M,{\tilde g})$ iff it holds on $(M,g)$. However, Theorem \ref{theorem4.6} uses the $\omega$-energy condition (\ref{eq4.9}), which only agrees with the strong energy condition when $\omega\to\infty$. For many matter models (e.g., dust with positive energy density), ${\rm tr}_gT<0$. When this is true, then the $\omega$-energy condition is, for any finite $\omega>-1$, a strictly weaker condition than the strong energy condition, and so in these circumstances Theorem \ref{theorem4.6} is stronger. We note as well that while Theorem \ref{theorem4.8} applies for all $\omega>-3/2$ (and for $\omega=-3/2$ as well, but the Brans-Dicke scalar equation doesn't permit $\omega=-3/2$), Theorem \ref{theorem4.6} applies only when $\omega\ge -1$, so in this sense Theorem \ref{theorem4.6} is obviously weaker. Both the distinction in the energy conditions and the distinction in applicable $\omega$ values arise because the Brans-Dicke scalar equation (\ref{eq4.3}) is used to bring equation (\ref{eq4.2}) into a form suitable for Theorem \ref{theorem4.1}, while only a conformal transformation of equation (\ref{eq4.2}) is used to obtain Theorem \ref{theorem4.8}.

Another distinction is that the positivity assumption (c) $H>0$ on mean curvature in Theorem \ref{theorem4.8} transforms under ${\tilde g}\mapsto g=\frac{1}{\varphi}{\tilde g}$ to $H>-\frac{3}{2\phi}\nabla_{\nu}\phi$, which differs from equation (\ref{eq4.13}) in assumption (d) of Theorem \ref{theorem4.6}.

Finally, we note that the rigidity results contained in Theorems \ref{theorem1.3} and \ref{theorem1.5} can also be applied to Brans-Dicke theory. To see this, note that Lemma \ref{lemma4.5} also holds if all the ``$>$'' signs in equations (\ref{eq4.10}--\ref{eq4.12}) are replaced by ``$\ge$'' signs. Then if the open inequality (\ref{eq4.13}) is replaced with the closed inequality $H\ge -\frac{1}{\varphi} \nabla_{\nu}\varphi$ and if $(M,g)$ is assumed to be past timelike geodesically complete, the time-reversed version of Theorem \ref{theorem1.3} implies that the past of $S$ splits as $((-\infty,0]\times S,-dt^2\oplus h)$ and $\varphi$ is independent of $t$ on the past of $S$. Thus the past of $S$ is static. Of course, nothing here forces $\varphi$ to be constant on $S$. Likewise, if the open inequality (\ref{eq4.15}) is replaced by the closed inequality $H_f(S) \ge 3 -\frac{1}{\varphi} \nabla_{\nu}\varphi$ and if we assume that the past timelike geodesics orthogonal to $S$ are past complete, then the time-reversed version of Theorem \ref{theorem1.5} implies that the past of $S$ is conformally static, being isometric to the warped product $((-\infty,0]\times S,-dt^2\oplus e^{2t}h)$, and $\varphi$ is constant on the past of $S$.

\end{document}